\documentclass{compositio}
\usepackage{amsmath}
\usepackage{enumerate}
\usepackage[arrow,matrix,tips,frame,color,line,poly]{xy}
\newtheorem{theorem}{Theorem}[section] 
\newtheorem{lemma}[theorem]{Lemma} 
\newtheorem{proposition}[theorem]{Proposition} 
\newtheorem{corollary}[theorem]{Corollary} 
\theoremstyle{definition} 
\newtheorem{definition}[theorem]{Definition} 
\theoremstyle{remark} 
\newtheorem{remark}{Remark}
\newtheorem{example}{Example}

\renewcommand{\exp}{\mathrm{exp}}
\newcommand{\wt}{\widetilde}
\newcommand{\caP}{\mathcal{P}}

\newcommand{\T}{\mathcal T}

\newcommand{\R}{\mathbb R}

\newcommand{\C}{\mathbb C}
\newcommand{\Z}{\mathbb Z}

\newcommand{\Ch}[1]{#1^\vee}

\newcommand\inv{^{-1}}
\newcommand{\g}{\mathfrak g}

\newcommand{\Fbar}{\overline F}
\renewcommand{\int}{\text{int}}

\newcommand{\rank}{\text{rank}}
\newcommand{\diag}{\text{diag}}
\newcommand{\Out}{\text{Out}}
\newcommand{\Int}{\text{Int}}
\newcommand{\Aut}{\text{Aut}}
\newcommand{\Ind}{\mathrm{Ind}}
\newcommand{\Cent}{\text{Cent}}
\newcommand{\Norm}{\text{Norm}}
\newcommand{\Stab}{\text{Stab}}

\newcommand{\Lie}{\mathrm{Lie}}
\newcommand{\LG}{\overset{L}{\vphantom{a}}\negthinspace G}

\begin{document}
\title{The Real Chevalley Involution}
\author{Jeffrey Adams}
\thanks{Supported in part by  National Science
Foundation Grants \#DMS-0967566, \#DMS-0968275 and \#DMS-1317523}
\begin{abstract}
The Chevalley involution of a
  connected, reductive algebraic group over an algebraically closed field takes every semisimple element
  to a conjugate of its inverse, and this involution is unique up to
  conjugacy. In the case of the reals we prove the existence of a real Chevally involution, 
which is defined over $\R$, takes every semisimple
  element of $G(\R)$ to a $G(\R)$-conjugate of its inverse, and is unique up to
  conjugacy by $G(\R)$. We derive some consequences, including an analysis of groups for which every 
irreducible representation is self-dual, and a calculation of the Frobenius Schur indicator for such groups. 
\end{abstract}
\maketitle

{\renewcommand{\thefootnote}{}}

\section{Introduction}
\label{s:introduction}
A Chevalley involution $C$ of a connected reductive group 
over an algebraically closed field satisfies $C(h)=h\inv$ for all $h$
in some Cartan subgroup of $G$.  Furthermore, $C$ takes any
semisimple\footnote{George Lusztig pointed out this is true for all
  elements \cite{lusztig_remarks}.} element to a conjugate of its
inverse.  Consequently, in characteristic $0$, for any algebraic representation $\pi$ of $G$,
$\pi^C$ is isomorphic to the contragredient $\pi^*$.

We are interested in the existence, and properties, of  rational
Chevalley involutions. 

\begin{definition}
\label{d:main}
Suppose $G$ is defined over a field $F$, and let $\Fbar$ be an algebraic closure of $F$.
\begin{enumerate}[(1)]
\item 
A Chevalley involution of $G(F)$ is the restriction of a Chevalley involution of $G(\Fbar)$ that is defined 
over $F$.
\item We say an involution of $G(F)$ is {\it dualizing}
if it takes every semisimple element of $G(F)$  to a $G(F)$-conjugate of its inverse.
\end{enumerate}
\end{definition}
We refer to a Chevalley involution of $G(\Fbar)$, which is defined over $F$, as an $F$-rational
Chevalley involution, or simply a rational Chevalley involution if $F$ is understood.

If $F$ is algebraically closed
every   Chevalley involution 
is dualizing, and any two are conjugate by an inner automorphism.
However, if $F$ is not algebraically closed, since not all Cartan subgroups
of $G(F)$ are conjugate, neither result is true in general. 
We are primarily interested in dualizing Chevalley involutions.

For certain classical groups, over any local field, 
there is a dualizing Chevalley involution by \cite[Chapitre 4]{mvw}.

Our main result is the existence of dualizing Chevalley
involutions in general when $F=\R$. Not all of these 
are conjugate  by an inner automorphism of $G(\R)$ (see Example \ref{e:notunique}). In order to have a
uniqueness result, we impose a further restriction.
A Cartan subgroup of $G(\R)$ is said to be fundamental if
it is of minimal split rank. Such a Cartan subgroup is the ``most compact''
Cartan subgroup, and is unique up to conjugation by $G(\R)$. 

\begin{theorem}
\label{t:main}
Suppose $G$ is defined over $\R$.
There is an involution $C$ of $G(\R)$ such that $C(h)=h\inv$ for all
$h$ in some fundamental Cartan subgroup of $G(\R)$. Any such
involution is the restriction of a rational Chevalley involution of
$G(\C)$, and is dualizing. Any two such involutions are conjugate by an inner
automorphism of $G(\R)$.
\end{theorem}

If $G$ is semisimple and simply connected this is due to Vogan 
\cite[Chapter I, Section 7]{borelWallach}.
The proof of the Theorem  is similar to the proof in \cite{borelWallach}.
See Remark \ref{r:deduce}.

If $G$ is simple all involutions (over local and finite fields) have been
classified by Helminck.  In particular Theorem \ref{t:main} can be read off
from \cite{hel88}, and similar results over other fields follow from
\cite{helminck_classification}.

\begin{definition}
We refer to an involution of $G(\R)$ satisfying the conditions of the
Theorem as a {\it fundamental} Chevalley involution of $G(\R)$.
\end{definition}

Since all fundamental Chevalley involutions are conjugate by an inner
automorphism of $G(\R)$, we may safely refer to {\it the} fundamental
Chevalley involution.

\begin{corollary}
\label{c:dual}
Suppose $\pi$ is an irreducible 
representation of $G(\R)$, and $C$ is the fundamental Chevalley involution.
Then $\pi^C\simeq \pi^*$.  
\end{corollary}
\medskip

Over a p-adic field it is not always obvious, at least to this author,
that there is a rational Chevalley involution,
not to mention a dualizing one.
In any event, the dualizing condition is quite restrictive.
For example, if $G(F)$ is split, there are
are many $G(F)$-conjugacy classes of involutions $C$ such that
$C(h)=h\inv$ for $h$  in  a split Cartan subgroup.
Most of these are not dualizing.
In fact, if $G$ is a split
exceptional group of type $G_2,F_4$ or $E_8$ over a p-adic field there is
{\it no}  dualizing involution. See Example \ref{e:E8}. 

\medskip

The map $\pi\rightarrow\pi^*$ defines an involution on L-packets.  
The main result of \cite{contragredient} is that, 
on the dual side, this involution is given by the Chevalley involution
of  $\LG$. See Section \ref{s:every}. 
It follows that there is an elementary condition for every L-packet to
be self-dual.

\begin{proposition}
\label{p:lpacket}
Every $L$-packet for $G(\R)$ is self-dual if and only if $-1\in W(G(\C),H(\C))$.
\end{proposition}
Here $H(\C)$ is any Cartan subgroup of $G(\C)$, and 
$W(G(\C),H(\C))$ is the (absolute) Weyl group $\Norm_{G(\C)}(H(\C))/H(\C)$.

Now  consider the finer question,  whether every irreducible representation of $G(\R)$ is self-dual.
Let $H_f(\R)$ be a fundamental Cartan subgroup of $G(\R)$, 
and let $W(G(\R),H_f(\R))=\Norm_{G(\R)}(H_f(\R))/H_f(\R)$.
\begin{theorem}
\label{t:every}
Every irreducible representation of $G(\R)$ is self-dual if and only
if $-1\in W(G(\R),H_f(\R))$.  
\end{theorem}

The condition is equivalent to: every semisimple element of $G(\R)$
is $G(\R)$-conjugate to its inverse.
We  give some information about when this condition holds 
in Section \ref{s:every}.
For example, suppose $G(\R)$ is connected,
$H_f(\R)$ is compact, and let $K(\C)$ be the complexification of a maximal compact subgroup of $G(\R)$. 
Then  $W(G(\R),H_f(\R))$ is the Weyl group of the root system of the connected, reductive group 
$K(\C)$. One can then look up this root system in a table, for example 
\cite{ov}, and check if it contains $-1$.

\begin{corollary}
\label{c:every}
Every irreducible  representation of $G(\R)$ is self-dual 
if and only if both of these  conditions hold:
\begin{enumerate}
\item[(a)] every irreducible representation of $K(\R)$ is self
dual,
\item[(b)]  $-1$ is in the absolute Weyl group $W(G(\C),H(\C))$. 
\end{enumerate}
If $G(\R)$ contains a compact Cartan subgroup then (a) implies (b).
\end{corollary}

It is perhaps surprising how common this is. For example:

\begin{theorem}
\label{t:everyadjoint}
If $-1\in W(G(\C),H(\C))$, and $G$ is of adjoint type, 
then every irreducible representation of $G(\R)$ is self dual.
\end{theorem}

For a more precise version, and some examples, see Section
\ref{s:every}, especially  Corollary \ref{c:practical}.

\medskip

We next give an application to Frobenius Schur indicators. If $\pi$ is an irreducible,
self-dual representation of $G(\R)$, the 
Frobenius Schur indicator $\epsilon(\pi)$ of $\pi$ is 
$\pm1$, depending on whether $\pi$ admits an invariant symmetric, or
skew-symmetric, bilinear form.
Write $\chi_\pi$ for the central character of $\pi$.
Let $\Ch\rho$ be one-half the sum of any set of positive co-roots.
Then $z(\Ch\rho)=\exp(2\pi i\Ch\rho)$ is in the center of $G(\R)$.

The Frobenius Schur indicator of a finite dimensional representation $\pi$
of $G(\C)$ is $\chi_\pi(z(\Ch\rho))$.  Under an assumption the same
holds for all irreducible (possibly infinite dimensional)
representations of $G(\R)$.

\begin{theorem}
Suppose every irreducible representation of $G(\R)$ is self-dual. 
Then, for any irreducible representation $\pi$,
$\epsilon(\pi)=\chi_\pi(z(\Ch\rho))$.

In particular the assumption holds if $-1\in W(G(\C),H(\C))$ and $G$ is of adjoint type (Theorem \ref{t:everyadjoint}), 
in which case every irreducible representation is orthogonal.
\end{theorem}

This paper is a complement to \cite{contragredient}, which considers
the action of the Chevalley involution on the dual group, and
its relation to the contragredient. See \cite[Remark 7.5]{contragredient}.

The author would like to thank Dipendra Prasad, Dinakar Ramakrishnan
and George Lusztig for very useful discussions. He also thanks the
referees for a number of helpful suggestions which have improved the
paper.

\section{Split Groups}
\label{s:split}

Throughout this paper
$G$ denotes a connected, reductive algebraic group, defined over a field $F$. 
We may identify $G$ with its points $G(\Fbar)$  over an algebraic closure of $F$.
In this section $F$ is arbitary; starting in the next section $F=\R$.
For background on algebraic groups see \cite{springer_book}, \cite{borel_algebraic_groups} 
or \cite{humphreys_algebraic_groups}.

We start by defining Chevalley involutions. This is well known,
although it isn't easy to find it stated in the terms we need. See
\cite[Section 2]{contragredient}.

By a {\it Chevalley involution} of $G=G(\Fbar)$ we mean an involution $C$
of $G$
satisfying $C(h)=h\inv$ for all $h$ in some Cartan subgroup $H$. 
Any two such involutions are conjugate by an inner automorphism.

Fix a pinning $\caP=(H,B,\{X_\alpha\})$:
$H\subset B$ are Cartan and Borel subgroups of $G$, respectively,
and (for $\alpha$ a simple root) $X_\alpha$ is in the $\alpha$-weight
space of $\Lie(H)$ acting on $\Lie(G)$.
Pinnings always exist, and are unique up to an inner automorphism; 
an inner automorphism fixes a pinning only if it is trivial.
For $\alpha$ a simple 
root let $X_{-\alpha}$ be the unique $-\alpha$-weight
vector satisfying $[X_\alpha,X_{-\alpha}]=\Ch\alpha\in\Lie(H)$. 

The choice of $\caP$ determines  a unique Chevalley involution $C$, 
satisfying $C(h)=h\inv$ ($h\in H$) and $C(X_\alpha)=X_{-\alpha}$ ($\alpha$ simple). 

\medskip

Now suppose $G$ is semisimple and simply connected, and $G(F)$ is split.
Generators and relations for $G(F)$   are given by 
\cite[Th\'eor\`em 3.2]{steinberg} (see \cite{steinbergCollected}).
The generators are $x_\alpha(u)$ for $\alpha$ a root, and $u\in F$, and these 
satisfy certain relations.
It is easy to check that the map $C(x_\alpha(u))=x_{-\alpha}(u)$ preserves
the defining relations of $G(F)$, and the resulting automorphism
satisfies $C(h)=h\inv$ for $h$ in a split Cartan subgroup.
 
\begin{lemma}
\label{l:split}
Suppose $G$ is semisimple and simply connected, and $G(F)$ is split.
Let $H(F)$ be a split Cartan subgroup. Then
there is a rational Chevalley involution
satisfying $C(h)=h\inv$ for all $h\in H(F)$.
\end{lemma}

\begin{remark}
The same result holds {\it a fortiori} for the (possibly) nonlinear covering group $\Delta$ of $G(F)$ of
\cite[Th\'eor\`em 3.1]{steinberg}, which is obtained by dropping some relations from 
those for $G(F)$. 
\end{remark}

This is a somewhat weak result.  Not every rational Chevalley
involution is dualizing, and not all dualizing Chevalley involutions are
conjugate by an inner automorphism of $G(F)$.
Both these facts are  illustrated by a simple example.
For $g\in G$, let $int(g)$ be conjugation by $g$. 

\begin{example}
\label{e:notunique}
Let $G(F)=SL(2,F)$.
Let $H_s(F)$ be the diagonal (split) Cartan subgroup.
Let $\sigma=\begin{pmatrix}0&1\\-1&0\end{pmatrix}$, 
and let $C=\int(\sigma)$.
Then $C(g)=\hskip-6pt\phantom{a}^tg\inv$ for all $g$, and in particular
$C(g)=g\inv$ for all $g\in H_s(F)$.

Suppose $g\ne\pm I$ is contained in an anisotropic Cartan subgroup 
$H_a(F)$.
Then  if $-1\not\in F^{*2}$, $C(g)$ is 
not conjugate   to $g\inv$ (in other words, $-1$ is not in the
Weyl group of $H_a(F)$).
Therefore $C$ is not dualizing. 

On the other hand let
$C'=\int(\diag(i,-i)\sigma)$.
Then $C'$ is rational and dualizing. Note that 
$C'$ is an outer automorphism of $G(F)$ unless $-1\in F^{*2}$.

Now replace $SL(2,F)$ with $G(F)=PGL(2,F)$.
Both $C,C'$ factor to inner automorphisms of  $G(F)$. 
Since every semisimple element of $G(F)$ is $G(F)$-conjugate to its
inverse, $C,C'$ are both dualizing. However it is easy to see that $C$ is
not conjugate to $C'$ by an inner automorphism of $G(F)$.
\end{example}

Surprisingly, even for split groups, which have rational Chevalley
involutions, there may be no
dualizing involution.
This is illustrated by the
following example, which was pointed out by D. Prasad \cite{prasadPreprint}.

\begin{example}
\label{e:E8}
Suppose $F$ is p-adic and $G(F)$ is the split  form of $G_2,F_4$
or $E_8$.  By Lemma \ref{l:split} there is a  Chevalley involution
$C$ of $G(F)$.
However, $G(F)$ has no dualizing involution. 

To see this, assume $\tau$ is a dualizing involution. Then
$\pi^\tau\simeq\pi^*$ for all
irreducible representations $\pi$.
Every  automorphism of $G(F)$ is inner (since $\Out(G)=1$ and $G$ is both simply connected and adjoint), 
so $\pi^\tau\simeq\pi$, 
and therefore every irreducible representation is self-dual.
However, there are irreducible  representations of $G(F)$ which
are not self-dual, coming from non-self dual cuspidal unipotent representations of the group over 
the residue field. 
\end{example}

\section{Real Chevalley Involutions}
\label{s:construction}

From now on we take $F=\R$, and 
we identify $G$ with its complex points $G(\C)$. We
recall some standard theory about real forms of $G$, 
in a form convenient for our applications.
For details see 
\cite[Section 5.1.4]{ov},
\cite{helgason_book},  \cite{knapp_beyond}
or \cite[Section 3]{algorithms}.
 
A real form $G(\R)$ of $G(\C)$ is the fixed points of an
anti-holomorphic involution. 
Each complex group has two distinguished real forms: the compact one,
and the split one.

For the compact real form, fix a pinning $\caP=(H,B,\{X_\alpha\})$
and define $\{X_{-\alpha}\}$ as at
the beginning of Section \ref{s:split}.
Let $\sigma_c$ be the unique antiholomorphic automorphism of $G$
satisfying
$\sigma(X_\alpha)=-X_{-\alpha}$. Then $G(\R)=G^{\sigma_c}$ is
compact, and $H(\R)\simeq S^{1}\times\dots \times S^1$ is a compact torus.

It is clear from the definitions that the Chevalley automorphism
$C=C_{\caP}$ commutes with $\sigma_c$.  Therefore $\sigma_s=C\sigma_c$
is an antiholomorphic involution of $G$. Furthermore $G(\R)=G^{\sigma_s}$ is
split: $H(\R)\simeq\R^{*}\times\dots\times\R^*$ is a split torus.

General real forms of $G$ may be classified either by antiholomorphic or
holomorphic involutions of $G$. The latter is provided by the theory
of the Cartan involution. 

In particular, there is a bijection
\begin{equation}
\label{e:bij}
\{\text{antiholomorphic involutions }\sigma\}/G
\leftrightarrow
\{\text{holomorphic involutions }\theta\}/G
\end{equation}
(the quotients are by conjugation by $\{\int(g)\mid g\in G\}$). 
If $\sigma$ is an antiholomorphic involution, after conjugating by $G$
we may assume it commutes with $\sigma_c$, and set
$\theta=\sigma\sigma_c$. The other direction is similar.

For example, by the preceding discussion, 
$C$ is the Cartan involution of the split real
form of $G$ (and the Cartan involution of the compact real form is the identity).

Suppose $\sigma\leftrightarrow\theta$, and $\sigma,\theta$ commute. 
Let $G(\R)=G^\sigma$, $K=G^\theta$, and $K(\R)=K\cap G(\R)=G(\R)^\theta=K^{\sigma}$. 
Then $K(\R)$ is a maximal compact subgroup of $G(\R)$, with
complexification $K$.
The relationship between these groups is illustrated by a diagram.
$$
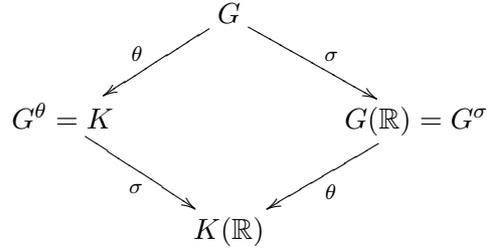
\xymatrix{
&G\ar[rd]^\sigma\ar[ld]_\theta\\
G^\theta=K\ar[rd]_\sigma&&G(\R)=G^\sigma\ar[ld]^\theta\\
&K(\R)
}
$$

\medskip

Write $\Aut(G),\Int(G)$ for the (holomorphic) automorphisms of $G$,
and the inner automorphisms, respectively. Let $\Out(G)=\Aut(G)/\Int(G)$ 
be the group of outer automorphisms.
We say an automorphism of $G$ is {\it distinguished} if it preserves
$\caP$.
The pinning $\caP$ defines an injective map 
$\Out(G)\overset s\hookrightarrow \Aut(G)$: 
$s(\phi)$ is the unique distinguished automorphism mapping to $\phi$.
If $G$ is semisimple the distinguished automorphisms embed into 
the automorphism group of the Dynkin diagram, and this is a bijection if
$G$ is simply connected or adjoint.

Now fix a holomorphic involution $\theta$ of $G$. 
Let $\delta$ be the image of $\theta$ under the map
$\Aut(G)\rightarrow\Out(G)\overset s\hookrightarrow \Aut(G)$, so
$\delta$ is distinguished.

\begin{lemma}
\label{l:theta}
After conjugating by $G$ we may assume 
\begin{equation}
\label{e:theta}
\theta=\int(h)\delta\quad\text{for some }h\in H^\delta
\end{equation}
(where the superscript denotes the $\delta$-fixed points). 
\end{lemma}

For example  suppose $\delta=1$, or equivalently $\theta\in\Int(G)$.
The assertion is that $\int(x)\circ\theta\circ\int(x\inv)=\int(h)$ for 
some $h\in H$. Since $\theta$ is inner write $\theta=\int(g)$ for some (semisimple) element $g\in G$.
The assertion is then: $\int(xgx\inv)=\int(h)$ for some $h\in H$. 
In other words this is the standard fact that any 
semisimple element is conjugate to an element of  $H$.

\smallskip

\begin{proof}
By the definition of $\delta$, $\theta=\int(g)\delta$ for some
semisimple element $g\in G$.

We claim $g$ is contained in a 
$\delta$-stable Cartan subgroup $H_1$.
Let $L$ be the identity component of $\Cent_G(g)$. 
Since $\theta$ is an involution, $\delta(g)=g\inv z$ for some $z\in Z$, 
and it is easy to see this implies $\delta(L)=L$. 
Take $H_1$ to be a $\delta$-stable Cartan subgroup of $M$. 
This contains $g$ and is clearly a Cartan subgroup of $G$.

Write $H_1=T_1A_1$ where $T_1$ (resp. $A_1$) 
is the identity component of $H_1^\theta$ (resp. $H_1^{-\theta}$). 
Since, for $h\in H_1$,  $h(g\delta)h\inv=h\delta(h\inv)g\delta$, 
we may assume the $A_1$ component of $g$ is trivial, i.e.  $h\in T_1$.

Let $K_\delta=G^\delta$. Use the subscript $0$ to indicate the identity component. 
Then $H_0^\delta=(H^\delta)_0$ is Cartan subgroup of
$K_{\delta,0}$. Now $T_1$ is a torus in $K_{\delta,0}$, and is therefore
$K_{\delta,0}$-conjugate to a subgroup of $H_0^\delta$. Therefore after conjugating by
$K_{\delta,0}$ we may assume $\theta=\int(h)\delta$ for $h\in H_0^\delta$. 
\end{proof}

With this choice of $\theta$, $H$ is defined over $\R$, and $H(\R)$ is a fundamental 
Cartan subgroup of $G(\R)$ (see the introduction).
We say $H$ is a fundamental Cartan subgroup of $G$ with respect to
$\theta$. 

For example
$\delta=1$ if and only if $H(\R)$ is compact.
For later use, we single out this class of groups.
We say $G(\R)$ is of equal rank if any of the following equivalent
conditions hold:
$G(\R)$ contains a compact Cartan subgroup; $H(\R)$ is compact;
 $\rank(K)=\rank(G)$; $\delta=1$; or  $\theta$ is an inner involution.

\medskip

We now give the proof of Theorem \ref{t:main}, which we break up into
steps. We first construct an involution of $G(\R)$, restricting to  $-1$ on a
fundamental Cartan subgroup.

\begin{lemma}
\label{l:step1}
Let $H(\R)$ be a fundamental Cartan subgroup. 
There is a rational Chevalley involution of $G$,
satisfying $C(h)=h\inv$ for all $h\in H(\R)$.
\end{lemma}

\begin{proof}
Choose $\theta$ corresponding to $\sigma$ by the bijection
\eqref{e:bij}.
By the Lemma, after conjugating $\sigma$ and $\theta$, we may assume 
$\theta=\int(h)\delta$, where $\delta$ is distinguished and
$h\in H^\delta$. 

Let $C=C_{\caP}$, the Chevalley involution defined by the splitting $\caP$, so
$C(h)=h\inv$ for $h\in H$. We claim $C$ commutes with $\sigma$.

First of all  $\theta$ and $\sigma_c$ commute.
On the one hand
\begin{subequations}
\renewcommand{\theequation}{\theparentequation)(\alph{equation}}  
\label{e:commute}
\begin{equation}
\begin{aligned}
(\theta\sigma_c)(X_\alpha)=\int(h)\delta(-X_{-\alpha})=-\int(h)(X_{-\delta\alpha})=-(\delta\alpha)(h\inv)X_{-\delta\alpha}
\end{aligned}
\end{equation}
and on the other
\begin{equation}
\begin{aligned}
(\sigma_c\theta)(X_\alpha)=\sigma_c(\int(h)X_{\delta\alpha})=\sigma_c((\delta\alpha)(h)X_{\delta\alpha})=-\overline{(\delta\alpha)(h)}X_{-\delta\alpha}.
\end{aligned}
\end{equation}
\end{subequations}
Since $\theta=\int(h)\delta$ is an involution,
$h\delta(h)\in Z(G)$ (here and elsewhere $Z$ denotes the center).
But $\delta(h)=h$, so $h^2\in Z(G)$. This
implies $\beta(h)=\pm1$ for all roots, so
$(\delta\alpha)(h\inv)=\overline{(\delta\alpha)(h)}$, and (a) and (b) are equal.

Therefore, by the discussion after \eqref{e:bij},  $\sigma=\theta\sigma_c$. 
Since $C$ commutes with $\sigma_c$ (see the beginning of this section), 
we just need to show that $C$ and $\theta$ commute.
This is similar to \eqref{e:commute}:
$(\theta C)X_\alpha=(\delta\alpha)(h\inv)X_{-\delta\alpha}$,
$(C\theta)X_\alpha=(\delta\alpha)(h)X_{-\delta\alpha}$, and these are equal since $h^2\in Z$.
\end{proof}

Now we show the Chevalley involution just constructed is dualizing
(Definition \ref{d:main}).

\begin{lemma}
Suppose $C$ satisfies the conditions of Lemma \ref{l:step1}. 
Then $C$ is dualizing, i.e. it takes every semisimple element of $G(F)$ to a
$G(F)$-conjugate of its inverse.
\end{lemma}

\begin{proof}
This is true for $g$ in the fundamental Cartan subgroup $H(\R)$. We
obtain the result on other Cartan subgroups using Cayley transforms.

We proceed by  induction, so change notation momentarily, and assume
$H$ is any $\theta$ and $\sigma$-stable Cartan subgroup, such that
$C(h)$ is $G(\R)$-conjugate to $h\inv$ for all $h\in H(\R)$.
Taking $h$ regular, we see there is $g\in\Norm_{G(\R)}(H(\R))$ such
that, if $\tau=\int(g)\circ C$, then $\tau|_{H(\R)}=-1$.

Suppose $\alpha$ is a root of $H$. 
Let $G_\alpha$ be the derived group of the centralizer of the kernel
of $\alpha$, and set $H_\alpha=H\cap G_\alpha$.
Thus $G_\alpha$ is locally isomorphic to $SL(2)$, 
and $H=\ker(\alpha)H_\alpha$.

Now assume $\alpha$ is a noncompact imaginary root, which amounts 
to saying that $G_\alpha$ is $\theta,\sigma$ stable, $G_\alpha(\R)$ is
split, and $H_\alpha(\R)$ is a compact Cartan subgroup of
$G_\alpha(\R)$.
Replace $H_\alpha$ with a $\theta,\sigma$-stable 
split Cartan subgroup $H'_\alpha$ of $G_\alpha$.
Since $\tau$ normalizes $G_\alpha$, and is defined over $\R$, 
$\tau(H'(\R))$ is another split Cartan subgroup of $G_\alpha(\R)$. 
Therefore we can find $x\in G_\alpha(\R)$ 
so that $x(\tau(h))x\inv=h\inv$ for all $h\in H'_\alpha(\R)$.

Let $H'=\ker(\alpha)H'_\alpha$. 
Then $(\int(x)\circ\tau)(h)=h\inv$ for all $h\in H'(\R)$.

Every Cartan subgroup of $G(\R)$ is obtained, up to conjugacy by
$G(\R)$, by a series of Cayley transforms from the fundamental Cartan
subgroup. The result follows.
\end{proof}

Finally, the uniqueness statement of Theorem \ref{t:main} comes down to the 
next Lemma.

\begin{lemma}
Suppose $\tau$ is an automorphism of $G(\R)$ such that the restriction of  $\tau$ to a 
fundamental Cartan subgroup $H(\R)$ is trivial. Then $\tau=\int(h)$ for some $h\in H(\R)$.
\end{lemma}

\begin{proof}
Since both $\R$ and $\C$ play a role here we write $G(\C)$ to
emphasize the complex group.
After complexifying, $\tau$ is an automorphism of $G(\C)$ which is
trivial on $H(\C)$.  It is well known that $\tau=\int(h)$ for some
$h\in H(\C)$ (for example see \cite[Lemma 2.4]{contragredient}).  It is enough
to show that $h\in H(\R)Z(G(\C))$.

Since $\tau$ normalizes $G(\R)$, $\sigma(h)=hz$ for
some $z\in Z(G(\C))$.  Writing $p$ for the map to the adjoint group, this
says $p(h)\in H_{ad}(\R)$. It is well known that $H_{ad}(\R)$ is connected
(this is where we use that $H(\R)$ is fundamental), so
the map $p:H(\R)\rightarrow H_{ad}(\R)$ is surjective.
Therefore we can find $h'\in H(\R)$ with $p(h')=p(h)$, i.e. $h=h'z\in H(\R)Z(G(\C))$. 
\end{proof}

\begin{lemma}
\label{l:uniqueness}
Any two automorphisms of $G(\R)$, restricting to $-1$ on 
a fundamental Cartan subgroup, are conjugate by an inner automorphism
of $G(\R)$.
\end{lemma}

\begin{proof}
Suppose $\tau,\tau'$ satisfying the conditions, with respect to a
fundamental Cartan subgroup $H(\R)$.
By the previous Lemma  $\tau'=\int(h)\circ\tau$ for some $h\in H(\R)$.
Since $H(\R)$ is connected, choose $x\in H(\R)$ with $x^2=h$. 
Then $\tau'=\int(x)\circ\tau\circ\int(x\inv)$. 
\end{proof}

This completes the proof of Theorem \ref{t:main}.

\begin{remark}
\label{r:deduce}
It is also possible to deduce Theorem \ref{t:main} from the special
case of \cite[Chapter I, Corollary 7.4]{borelWallach}, 
which is essentially about the Lie algebra. 
According to this result (actually, its proof), 
if $G(\C)$ is semisimple and simply connected, there is a
rational Chevalley $C$ involution of $G(\C)$, whose restriction to
$G(\R)$ is dualizing. 

Since $C$ acts by inverse on the center of $G(\C)$, it preserves
any subgroup of the center, and therefore factors 
to any quotient of $G(\C)$. 
Similarly, any 
complex reductive group is a quotient of a simply connected semisimple
group and a torus, and a similar argument holds in this case.
\end{remark}

\section{Groups for which every representation is self-dual}
\label{s:every}

We first consider the elementary question of when every L-packet is
self-dual (Proposition \ref{p:lpacket}).

Fix a real form $G(\R)$ of $G(\C)$,   choose $\theta$  as usual, and let
$K(\C)=G(\C)^\theta$ (see Section \ref{s:construction}).  Let $\g=\Lie(G(\C))$.
By an irreducible representation of $G(\R)$ we  mean 
an irreducible $(\g,K(\C))$-module, or equivalently 
an irreducible admissible representation of $G(\R)$ on a complex Hilbert space.
See \cite[Section 0.3]{vogan_green}.

We now identify $G$ with $G(\C)$, $K$ with $K(\C)$, and similarly for other.
We will always write $\R$ to indicate a real group.

\medskip

\begin{proof}[Proof of Proposition \ref{p:lpacket}]
Suppose an L-packet  $\Pi$
is defined by an  admissible homomorphism  $\phi:W_\R\rightarrow \LG$.
By \cite[Theorem 1.3]{contragredient} the contragredient L-packet 
corresponds to $C\circ\phi$, where $C$ is the Chevalley automorphism
of $\LG$.
Therefore every L-packet is self-dual if and only if this action is trivial, up
to conjugation by $\Ch G$, i.e. the Chevalley automorphism is inner
for $\Ch G$. This is the case if and only if $-1\in W(\Ch G,\Ch
H)\simeq W(G,H)$.
\end{proof}

\begin{remark}
\label{r:-1}
By the classification of root systems, $-1$ is in  the Weyl group of an irreducible
root system if and only if it  is of type
$A_1$, $B_n$, $C_n$, $D_{2n}$, $F_4$, $G_2$, $E_7$,  or $E_8$. It is worth noting that 
if $G$ is simple and simply connected, $-1\in W(G,H)$ if and only if
$Z(G)$ is an elementary two-group (one direction is obvious, and  the other is case-by-case). 
\end{remark}

We are interested in real groups $G(\R)$ for which
every irreducible representation  is self-dual.
By Proposition \ref{p:lpacket} an obvious necessary condition is  $-1\in W(G,H)$.
We first prove Theorem \ref{t:every}, which gives 
a necessary and sufficient condition,  and then give more
detail in some special cases.

Let $H_f$ be the centralizer in $G$ of a Cartan subgroup of $K_0$ (the
subscript indicates identity component).  Let $H_K=H_f\cap K$. This is
an abelian subgroup of the (possibly disconnected) group $K$, and
$H_{K,0}=H_K\cap K_0$ is a Cartan subgroup of $K_0$.  Then $H_f$ is a
fundamental Cartan subgroup of $G$ with respect to $\theta$ (see
\cite[Definition 3.1]{khat}).  For example, choose $\theta$ as in
Lemma \ref{l:theta}. Then $H_f$ is the fixed Cartan subgroup of the
pinning $\caP$.

\medskip

\begin{proof}[Proof of Theorem \ref{t:every}]
Using standard facts about characters of representations, viewed as
functions on the regular semisimple elements, it is easy
to see that every irreducible representation  is self-dual if and only if
\begin{equation}
\label{e:conj}
\text{ every regular semisimple element is  $G(\R)$-conjugate to its inverse}.
\end{equation}
Assume $-1\in  W(G(\R),H_f(\R))$,
so there is an inner automorphism $\tau$ of $G(\R)$ acting by $-1$ on
$H_f(\R)$.
By Theorem \ref{t:main}, if $g$ is semisimple, $\tau(g)$ is
$G(\R)$-conjugate to $g\inv$. Since $\tau$ is inner this gives
\eqref{e:conj}.

Conversely suppose \eqref{e:conj} holds. Let $h$ be a regular element
of $H_f(\R)$. Then $h\inv=xhx\inv$ for some $x\in G(\R)$, and by
regularity 
$x$ normalizes $H_f(\R)$. Therefore $-1\in W(G(\R),H_f(\R))$.
\end{proof}

\begin{subequations}
\renewcommand{\theequation}{\theparentequation)(\alph{equation}}  
\label{e:WKH}

It is helpful to state this result in terms of the complex group $K$,
rather than the real group $G(\R)$.
The groups
\begin{equation}
W(G(\R),H_f(\R))=\Norm_{G(\R)}(H_f(\R))/H_f(\R).
\end{equation}
and 
\begin{equation}
W(K,H_f)=\Norm_K(H_f)/H_K
\end{equation}
are isomorphic. We reiterate that $K, H_f$ and $H_K$ are complex.
Also consider 
\begin{equation}
W(K,H_K)=\Norm_K(H_K)/H_K.
\end{equation}
\end{subequations}
This is defined solely in terms of $K$; the difference between (b) and (c)
is whether we consider an element to be an automorphism of $H_f$ or $H_K$ (see the next Remark). 
This is also isomorphic to
(a) and (b), and is useful in computing these groups.  

Some care is required here due to the fact that $K$, equivalently
$G(\R)$, may be disconnected.  If $K$ is connected then $W(K,H_K)$ is
the Weyl group of the root system of $H_K$ in $K$, but otherwise
$W(K,H_K)$ may not be the Weyl group of a root system.

A key role is played by the condition $-1\in W(K,H_f)$.
We need to keep in mind 
the following dangerous bend
concerning the meaning of $-1$.

\begin{remark}
\label{r:dangerous}
Suppose $-1\in W(K,H_K)$. By definition this means there is an element
$g\in \Norm_K(H_K)$ such that $ghg\inv=h\inv$ for all $h\in H_K$. 
However, although $g$ normalizes $H_f$,
it is not necessarily the case that $ghg\inv=h\inv$ for all $h\in
H_f\supset H_K$.

In other words, if $\rank(K)\ne \rank(G)$, $-1\in W(K,H_K)$ does not
imply $-1\in W(K,H_f)$, even though these two groups are isomorphic.

On the other hand $-1\in W(G(\R),H_f(\R))$ if and only if $-1\in W(K,H_f)$.
\end{remark}

\begin{example}
Let $G=SL(3,\C)$, $G(\R)=SL(3,\R)$. 
Then $-1\not\in W(G,H_f)$, so {\it a fortiori} $-1\not\in W(K,H_f)$.
On the other hand $K=SO(3,\C)$, 
$W(K,H_K)$ is the Weyl group of type $A_1$, 
and $-1\in W(K,H_K)$.

We can choose $H_f=\{(z,w,\frac1{zw})\}\mid z,w\in\C^*\}$, 
and $H_K=\{(z,\frac1z,1)\}\subset H_f$.
The  nontrivial Weyl group element of $W(K,H_K)$ acts by exchanging the first 
two coordinates.
This acts by inverse on $H_K$, but not $H_f$.
\end{example}

If $K$ is
connected, it is an elementary  root system check to determine if $-1\in
W(K,H_K)$ (see Remark \ref{r:-1}). In the equal rank case this is all that is needed, although
in the unequal rank case some care is required to determine if
$-1\in W(K,H)$.

By the isomorphism of \eqref{e:WKH}(a) and (b),
Theorem \ref{t:every} can be stated in terms of $W(K,H_f)$.

\begin{corollary}
\label{c:everyalt}
Every irreducible representation of $G(\R)$ is self-dual if and only if 
$-1\in W(K,H_f)$.
\end{corollary}

Next we  prove Corollary \ref{c:every}, which gives another condition, 
in terms of $K$, 
for every representation of $G(\R)$ to be self dual.

\medskip

\begin{proof}[Proof of Corollary \ref{c:every}]
Every irreducible representation $\mu$ of $K(\R)$ is the unique lowest $K(\R)$-type of
an   irreducible representation $\pi$ of $G(\R)$ \cite[Theorem 1.2]{khat}. 
Since the lowest $K(\R)$-type of $\pi^*$ is $\mu^*$, $\pi\simeq \pi^*$
implies $\mu\simeq\mu^*$. This proves one direction. 

Conversely, by Corollary \ref{c:everyalt} we need to show every
irreducible representation of $K(\R)$, equivalently $K=K(\C)$, is self-dual implies $-1\in W(K,H_f)$.

We first show that  $-1\in W(K,H_K)$ and $-1\in
W(G,H_f)$ implies
$-1\in W(K,H_f)$.
This is obvious if $H_K=H_f$ (the equal rank case). Otherwise (here we
need the assumption that $-1\in W(G,H_f)$) choose $g\in G$ such that
$ghg\inv=h\inv$ for all $h\in H_f$. Also choose $k\in K$ satisfying
$khk\inv=h\inv$ for all $h\in H_K$. Then $gk\inv\in\Cent_G(H_K)=H_f$.
This implies $khk\inv=h\inv$ for all $h\in H_f$.

So it is enough to show that if  every irreducible representation of $K$ is
self-dual then $-1\in W(K,H_K)$.
If $K$ is connected this follows from  Corollary 
\ref{c:everyalt} applied to $K$.

For $\lambda\in X^*(H_{K,0})$  (the algebraic characters of the torus $H_{K,0}$)
let $\pi_\lambda$ be the irreducible
representation of $K_0$ with extremal weight $\lambda$. 
Then $\pi_\lambda^*=\pi_{-\lambda}$. 

Consider the induced representation
$I=\Ind_{K_0}^K(\pi_\lambda)$.
The restriction of $I$ to $K_0$ contains $\pi_\lambda$.
Since $I$ is self-dual by hypothesis, this restriction also contains
$\pi_{-\lambda}$.

It is easy to see
that every extremal weight of the
restriction of this representation to $K_0$ is $W(K,H_K)$-conjugate to
$\lambda$ (choose representatives of $K/K_0$ in $\Norm_K(H_K)$, and
use the fact that $K_0$ is normal in $K$). Therefore $-\lambda$ is
$W(K,H_K)$-conjugate to $\lambda$.  Taking $\lambda$ generic
this implies $-1\in W(K,H_K)$.

If every irreducible representation of $K$ is self-dual then $-1\in
W(K,H)$. 
If $\rank(G)=\rank(K)$ this implies $-1\in W(G,H_f)$, giving the final assertion.
\end{proof}

\begin{remark}
Here is an example of an unequal rank group for which Condition (a) in \eqref{c:every} holds, but not (b).
Take  $G(\R)=SL(2n+1,\R)$, $K=SO(2n+1,\C)$.
Then $-1\in W(K,H_K)$, and every irreducible representation of
$SO(2n+1,\C)$ is self-dual. However this is not 
the case (for example for minimal principal series) for $SL(2n+1,\R)$, 
since $-1\not\in W(G,H_f)$. 
\end{remark}

Here is a practical way to determine if every irreducible
representation of $G(\R)$ is self-dual.

First assume $G(\R)$ is of equal rank (see the discussion after Lemma
\ref{l:theta}).
Then $\theta$ is inner, so write  $\theta=\int(x)$ for some $x\in G$, 
with $x^2\in Z(G)$. 

Assume for the moment that $-1\in W(G,H_f)$ (recall $G=G$ and $H_f$ are complex);  this implies $Z(G)$ is an
elementary two-group. We say the real form defined by $\theta$ is {\it pure} if $x^2=1$. 
Since $Z(G)$ is a two-group, this condition is independent of the choice of $x$ such that
$\theta=\int(x)$.
(In other words, although  purity is typically only well-defined  
as a property of {\it strong} real forms \cite[Definition
5.5]{algorithms},
it is a  well-defined property of real forms provided $-1\in W(G,H_f)$.)
Every real form is pure if $G$ is adjoint. 

\begin{corollary}
\label{c:practical}
Assume $G(\R)$ is simple. 
Every irreducible representation of $G(\R)$ is self-dual if and only
if both of these  conditions hold:
\begin{enumerate}
\item[(a)]   $-1\in W(G,H_f)$ 
\item[(b)]  if $G(\R)$ is of equal rank, it is a pure real form.
\end{enumerate}
\end{corollary}

\begin{proof}
First assume we are in the equal rank case. 
By Theorem \ref{t:main} we have to show
\begin{equation}
\label{e:practical}
-1\in W(G,H_f), x^2=1\Leftrightarrow -1\in W(K,H_f).
\end{equation}
After conjugating by $G$  we may assume $x\in H_f$.
Suppose $g\in G$ satisfies $ghg\inv=h\inv$ for all $h\in H_f$.
Then $\theta_x(g)=xgx\inv=x(gx\inv g\inv)g=x^2g$.
Therefore $g\in K$ if and only if $x^2=1$. 

Now suppose $G(\R)$ is not of equal rank. We have to show
\begin{equation}
-1\in W(G,H_f)\Leftrightarrow -1\in W(K,H_f).
\end{equation}
The implication $\Leftarrow$ is obvious.

First assume
$G(\R)=G_1(\C)$, i.e. a complex group, viewed as a real group by
restriction of scalars. 
Then, if $H_1$ is a Cartan subgroup of $G_1$, 
$G=G_1\times G_1, H_f=H_1\times H_1, K=G_1^\Delta$ (embedded diagonally).
It follows immediately that 
$-1\in W(K,H_f)$   if and only if $-1\in W(G,H_f)$.

Finally assume $G(\R)$ is unequal rank, but not complex.
Then 
$G$ is of type $A_n$ ($n\ge 2$),
$D_n$ or $E_6$. But then $-1\in W(G,H_f)$ only in type $D_{2n}$. 
This leaves only the groups locally isomorphic to $SO(p,q)$ with $p,q$
odd and $p+q=0\pmod 4$. 

Let $G(\R)=Spin(p,q)$ with $p+q=4n$. 
It is enough to show $-1\in W(K,H_f)$, 
since $W(K,H_f)$ is, if anything, larger if $G$ is not simply
connected. 
Note that $K$ is connected, of type $B_r\times B_s$, and $-1\in
W(K,H_K)$. 
The only remaining issue is to check that $-1\in W(K,H_f)$; here
$\rank(H_f)=\rank(H_K)+1$. This is a straightforward check.
It essentially comes down to the case of $Spin(3,1)$, for which it is 
easy to see, since $Spin(3,1)\simeq SL(2,\C)$.
\end{proof}

Theorem \ref{t:everyadjoint} is a special case of Corollary \ref{c:practical}. 

\begin{proof}[Proof of Theorem \ref{t:everyadjoint}]
By assumption $-1\in W(G(\C),H(\C))$ so (a) of Corollary \ref{c:practical} holds.
On the other hand (b) holds since every real form of an adjoint group is pure.
By the Corollary every irreducible representation of $G(\R)$ is self-dual.  
\end{proof}

With a little effort we can deduce the following list from Corollary \ref{c:everyalt} and Corollary
\ref{c:practical}.

First assume $G$ is simple, and $G(\R)$ is equal rank.
If $G$ is adjoint it is only a question of whether $-1\in W(G,H_f)$.
If $G$ is simply connected we need to check
if $-1$ is in the Weyl group of the root system of $K$, which is easy, 
for example by the tables \cite[pp. 312-317]{ov}.
This leaves only the intermediate
groups of type $D_n$, which require some case-by-case checking.

In the unequal rank case, we only need to consider 
complex groups, and (up to isogeny) $SO(p,q)$ with
$p,q$ odd. 
\medskip

Suppose $G(\R)$ is simple. Then every irreducible representation of
$G(\R)$ is self-dual if and only if $G(\R)$ is on the following list
(see below for terminology in type $D_{2n}$). 

\begin{enumerate}[(1)]
\item $A_n$: $SO(2,1)$, $SU(2)$  and $SO(3)$.
\item $B_n$: Every real form of the adjoint group, $Spin(2p,2q+1)$
  ($p$ even).
\item $C_n$: Every real form of the adjoint group,  all $Sp(p,q)$.
\item $D_{2n+1}$: none.
\item $D_{2n}$, equal rank:
$Spin(2p,2q)$ ($p,q$ even);
all $SO(2p,2q)$ ($p+q=2n$)
$\overline{SO}(2p,2q)$ ($p,q$ even);
$\overline{SO}^*(4n)$ when disconnected; 
all adjoint groups: $PSO(2p,2q)$ ($p+q=2n$) and  $PSO^*(4n)$.

\item $D_{2n}$, unequal rank: all real forms, i.e. all groups locally isomorphic to
  $SO(2p+1,2q+1)$ ($p+q$ odd).
\item $E_6$: none.
\item $E_7$: Every real form of the adjoint group, the simply  connected compact group.
\item $G_2,F_4,E_8$: every real form.
\item complex groups of type $A_1,B_n,C_n,D_{2n},G_2,F_4,E_7,E_8$ (see
  Remark \ref{r:-1}).
\end{enumerate}
In type $D_{2n}$ let 
$\overline{SO}(4n,\C)$ denote the group $Spin(4n,\C)/A$
where $A\simeq\Z/2\Z$ is not fixed by the outer automorphism of $Spin(4n,\C)$. 
For each $p+q=4n$ this group has a real form denoted 
$\overline{SO}(p,q)$ (locally isomorphic to $SO(p,q)$). 
Also it has two 
subgroups locally isomorphic to $SO^*(4n)$, which we denote $\overline{SO}^*(4n)$. 
These are not isomorphic: one of them is connected, and the other is not.
\section{Frobenius Schur Indicators}

Suppose $\pi$ is an irreducible self-dual representation of a group $G$.
Choosing an isomorphism $T:\pi\rightarrow\pi^*$, $\langle v,w\rangle:=T(v)(w)$ 
is a non-degenerate, invariant, bilinear form, unique up to scalar. 
It is either symmetric or skew-symmetric. The 
Frobenius Schur indicator $\epsilon(\pi)$ of $\pi$ is defined to be $1$ or $-1$, accordingly. 
It is of some interest to compute this invariant. For example see \cite{prasadRamakrishnan}.

Now suppose $G$ is a connected, reductive complex group.
It is well known that if $\pi$ is a self-dual,  finite dimensional representation of $G$ 
 $\epsilon(\pi)$
is given by a particular value of its central character \cite[Ch. IX,
\S7.2, Proposition 1]{bourbakiLieGroupsLieAlgebras7-9}.  
Here is an elementary proof. This is a  refinement of 
one of the proofs of 
\cite[Section 1, Lemma 2]{prasadSelfDual};
we use the Tits group to identify the central element in question.

Let
$\Ch\rho$ be one-half the sum of any set of positive co-roots, and set
\begin{equation}
\label{e:z}
z(\Ch\rho)=\exp(2\pi i\Ch\rho).
\end{equation}
Not only is $z(\Ch\rho)$  central in $G$,  it is
fixed by every automorphism of $G$. 
In particular $z\in Z(G(\R))$ for any real form of $G$.
If it is necessary to specify the group in question we will write
$z(\Ch\rho_G)$. 

\begin{lemma}
\label{l:w0}
Let $w_0$ be the long element of $W(G,H)$ (with respect to any set of positive roots). 
There is a representative $g\in\Norm_G(H)$ of $w_0$ satisfying
$g^2=z(\Ch\rho)$. Furthermore, if $w_0=-1$, this holds for any
representative of $w_0$.
\end{lemma}

\begin{proof}
We use the Tits group.
Fix a pinning $\caP=(H,B,\{X_\alpha\})$ for $G$
(see Section \ref{s:split}). This defines the Tits group $\T$, a subgroup
of $\Norm_G(H)$ mapping surjectively to $W(G,H)$. Every element $w$ of the
Weyl group has a canonical inverse image  $\sigma(w)\in\T$.
See \cite[Section 5]{contragredient}.

Let $g=\sigma(w_0)$. By \cite[Lemma 5.4]{contragredient}, $g^2=z(\Ch\rho)$.
Any other representative is of the form $hg$ for some $h\in H$.
If $w_0=-1$ then $(hg)^2=h(ghg\inv)g^2=(hh\inv)g^2=g^2$.
\end{proof}

\begin{lemma}
\label{l:schurfinite}
Assume $G$ is a connected, reductive complex group.
Suppose $\pi$ is an irreducible, finite dimensional,  self-dual representation of $G$. 
Let $\chi_\pi$ denote the central character of $\pi$. Then 
\begin{equation}
\epsilon(\pi)=\chi_\pi(z(\Ch\rho)).
\end{equation}
\end{lemma}

\begin{proof}
For any vectors $u,w$ in the space $V$ of $\pi$ we have
\begin{subequations}
\renewcommand{\theequation}{\theparentequation)(\alph{equation}}  
\label{e:schur}
\begin{equation}
\langle u,w\rangle=\epsilon(\pi)\langle w,u\rangle.
\end{equation}
Suppose $g\in G,g^2\in Z(G)$, and $v\in V$. Set $u=\pi(g^2)v,w=\pi(g)v$: 
\begin{equation}
\begin{aligned}
\chi_\pi(g^2)\langle v,\pi(g)v\rangle&=
\langle \pi(g^2)v,\pi(g)v)\quad\text{(since $g^2$ is central)}\\
&=\langle \pi(g)v,v\rangle\quad\text{(by invariance)}\\
&=\epsilon(\pi)\langle v,\pi(g)v\rangle\quad\text{(by (a)).}
\end{aligned}
\end{equation}
We conclude
\begin{equation}
g^2\in Z(G),\, \langle v,\pi(g)v\rangle\ne 0\Rightarrow \epsilon(\pi)=\chi_\pi(g^2).
\end{equation}
\end{subequations}

Fix a Cartan subgroup $H$, and for $\lambda\in X^*(H)$ write
$V_\lambda$ for the corresponding weight space.  
It is easy to see $\langle V_\lambda,V_{-\lambda}\rangle\ne 0$. 

Let $\lambda$ be the highest weight, so $V_\lambda$ is
one-dimensional. Let $w_0$ be the long element of the Weyl group. 
Then $\pi^*$ has highest weight $-w_0\lambda$; since $\pi$ is
self-dual this implies $-\lambda=w_0\lambda$. 

Choose $g\in \Norm_G(H)$ as in  Lemma \ref{l:w0}, so $g^2=z(\Ch\rho)$,  and $0\ne v\in V_\lambda$. 
Then $\pi(g)(v)\in V_{-\lambda}$. Since $V_{\pm\lambda}$ are
one-dimensional $\langle v,\pi(g)v\rangle\ne 0$, so apply \eqref{e:schur}(c).
\end{proof}

We now consider the Frobenius Schur indicator for infinite dimensional
representations. The basic technique is the following elementary observation, 
which appears in \cite{prasadRamakrishnan}. 

Suppose $H\subset G$ are groups,  $\pi$ is a self-dual representation
of $G$, $\pi_H$ is a self-dual representation of $H$, and 
$\pi_H$ occurs with multiplicity one in $\pi|_H$.
Then $\epsilon(\pi)=\epsilon(\pi_H)$.
We first apply this to $G$ and $K$, and later to $K$ and its identity component.

The next  Lemma is a special case of the main result of this section
(Theorem \ref{t:schur}), but it is worth stating separately since 
it clearly illustrates the main idea.

We continue to assume $G$ is a connected reductive complex group. 
Fix a real form $G(\R)$,  a corresponding Cartan
involution $\theta$, and let $K=G^\theta$.

\begin{lemma}
\label{l:schurconnected}
Suppose every irreducible representation of $G(\R)$ is
self-dual.
Also assume $G(\R)$ is
connected. 
If $\pi$ is an irreducible representation then
$\epsilon(\pi)=\chi_\pi(z(\Ch\rho))$. 
\end{lemma}

\begin{proof}
By Corollary \ref{c:everyalt}, 
the self-duality assumption implies $-1\in
W(K,H_f)$. So $-1\in W(K,H_K)$ and this implies every $K$-type is self-dual (since
$G(\R)$, and therefore  $K$, is
connected).

Let $\mu$ be a lowest $K$-type of $\pi$. 
Then $\mu$ has multiplicity one, and is self-dual,
so by the comment above $\epsilon(\pi)=\epsilon(\mu)$.
By Lemma \ref{l:schurfinite}, $\epsilon(\mu)=\chi_\mu(z(\Ch\rho_K))$,
where  $z(\Ch\rho_K)$ is defined by \eqref{e:z} applied to $K$.
Write $\Ch\rho_G$ in place of  $\Ch\rho$.
Let $g\in\Norm_G(H_f)$ be a representative of $-1\in W(G,H_f)$, so by Lemma \ref{l:w0}
$g^2=z(\Ch\rho_G)$.
Now view $g$ as a representative of $-1\in W(K,H_f)$, in which case 
(by Lemma \ref{l:w0} applied to  to $K$) we see $g^2=z(\Ch\rho_K)$.

Therefore 
$z(\Ch\rho_G)=z(\Ch\rho_K)$, and since $z(\Ch\rho_G)\in Z(G)$, $\chi_\mu(z(\Ch\rho_G))=\chi_\pi(z(\Ch\rho_G))$, independent 
of $\mu$. Thus
$$
\epsilon(\pi)=\epsilon(\mu)=\chi_\mu(z(\Ch\rho_K))=\chi_\mu(z(\Ch\rho_G))=\chi_\pi(z(\Ch\rho_G)).
$$
\end{proof}

A crucial aspect of the proof is that, for $K$ connected, 
$-1\in W(K,H_f)$ implies  $z(\Ch\rho_G)=z(\Ch\rho_K)$.
We need the surprising fact that
this is true without the first assumption.

\begin{lemma}
\label{l:zrhoK}
Suppose $G$ is a connected, reductive complex group, $\theta$ is a Cartan involution,
$K=G^\theta$ and  $H_f$ is a fundamental Cartan subgroup. 
Assume $-1\in W(K,H_f)$. Then $z(\Ch\rho)=z(\Ch\rho_K)$.
\end{lemma}

This is a bit subtle, as 
 a simple example shows.

\begin{example}
Let $G(\R)=SL(2,\R)$, so $K=SO(2,\C)$.
Then $-1\not\in W(K,H_f)$, and $-I=z(\Ch\rho)\ne z(\Ch\rho_K)=I$. 

On the other hand suppose $G(\R)=PSL(2,\R)=SO(2,1)$. 
Then $K=O(2,\C)$, so $-1\in W(K,H_f)$, and now
$I=z(\Ch\rho)=z(\Ch\rho_K)$. 
\end{example}

\begin{proof}
We may  assume $G(\R)$ is simple.

First assume $G(\R)$ is equal rank.
Recall (see the discussion in Section \ref{s:split}) $K=\Cent_G(x)$
for  some $x\in H_f$. We will show $x$ is of a particular form.
We need a short digression on the 
Kac classification of real forms. For details see
\cite{ov},\cite{helgason_book}.

Let $\wt D$ be the 
extended Dynkin diagram  for $G$, with nodes  ${0,\dots,
m}$; roots $\alpha_0,\dots, \alpha_m$ ($-\alpha_0$ is the highest
root); and  labels $n_0=1,n_1,\dots, n_m$ (the multiplicity of the
root in the highest root). 
The Dynkin diagram of $K$ is 
obtained from $\wt D$ by deleting 
node $j$ with label $2$, or nodes ${j,k}$ with label $1$. 
In the second case, without loss of generality, we may assume $k=0$, 
so both cases may be combined, as specifying
a single node $j$ with label $n_j=1$ or
$2$.

Let $\Ch\lambda_j$ be the $j^{th}$ fundamental weight for $G$. 
Then we can take
$x=\exp(\pi i\Ch\lambda_j)$.

\begin{subequations}
\renewcommand{\theequation}{\theparentequation)(\alph{equation}}  
Now set $N=\sum_{i=0}^m n_i$ and  let
\begin{equation}
c=
\begin{cases}
\frac N2&n_j=2\\
N-1&n_j=1.  
\end{cases}
\end{equation}
Except in type $A_{2n}$, which is ruled out since $-1\in W(G,H_f)$, $N$ is even, so $c\in \Z$.

It is  an exercise in root  systems to see that
\begin{equation}
\Ch\rho_G-\Ch\rho_K=c\Ch\lambda_j.
\end{equation}
(For  $i\ne 0,j$, both sides are $0$ when paired with 
$\alpha_i$, so this amounts to computing the pairing with 
$\alpha_0$ and $\alpha_j$.)
Therefore
\begin{equation}
x=\exp(\frac{\pi i}c(\Ch\rho_G-\Ch\rho_K)).
\end{equation}
\end{subequations}
Then  $x^{2c}=z(\Ch\rho_G)/z(\Ch\rho_K)$.

By  \eqref{e:practical} we have:
\begin{equation}
-1\in W(K,H_f)\Leftrightarrow x^2=1\Rightarrow x^{2c}=1
\Rightarrow z(\Ch\rho_G)=z(\Ch\rho_K).
\end{equation}

A similar, but more elaborate, argument holds in the unequal rank
case. 
Instead, we proceed in a more case-by-case fashion.
If $G(\R)$ is complex, then $K$ is connected, and we have already
treated this case (see the proof of Lemma \ref{l:schurconnected}).
Since $-1\in W(K,H_f)$ every representation of $G(\R)$ is self-dual. 
Consulting the list at the end of the previous section, this leaves
only type $D_{2n}$.
If $G$ is simply connected then by a case-by-case check (assuming unequal rank),
$-1\in W(K,H_f)$, and $K$ is connected, so again we have
$z(\Ch\rho_G)=z(\Ch\rho_K)$. 
The result is then true {\it a fortiori} if $G$ is not simply
connected. This completes the proof.
\end{proof}

We also need a generalization of Lemma \ref{l:schurfinite}.

\begin{lemma}
\label{l:schurfinite2}
Assume $G$ is a connected, reductive complex group.
Let $G^\dagger=G\rtimes\langle\delta\rangle$ where $\delta^2\in Z(G)$ and $\delta$
acts on $G$ by a Chevalley involution.

Every irreducible finite dimensional representation $\pi^\dagger$ of 
$G^\dagger$ is self-dual, and 
if $\pi$ is an irreducible constituent of $\pi^\dagger|_G$, then
\begin{equation}
\epsilon(\pi^\dagger)=
\begin{cases}
\epsilon(\pi)&\pi\simeq\pi^*\\
\chi_\pi(\delta^2)&\pi\not\simeq\pi^*.
\end{cases}
\end{equation}
\end{lemma}

\begin{proof}
The restriction of $\pi^\dagger$ is irreducible if and only if
$\pi\simeq\pi^\delta$. 
Since $\delta$ acts by the Chevalley involution, this is equivalent to
$\pi\simeq\pi^*$. 

If $\pi\simeq\pi^*$ the result is clear. Otherwise, let $\lambda$ be
the highest weight of $\pi$. Then $\pi^\delta$ has extremal weight
$-\lambda$, i.e. highest weight $-w_0\lambda$ where $w_0$ is the long
element of the Weyl group. Since $\pi\not\simeq\pi^*$,
$-w_0\lambda\ne\lambda$, so the $\lambda$-weight space of
$\pi^\dagger$ is one-dimensional. The proof of Lemma
\ref{l:schurfinite} now carries through using $\delta$, which interchanges
the $\lambda$ and $-\lambda$ weight spaces of $\pi^\dagger$.
\end{proof}

We need to consider finite dimensional representations of the possibly
disconnected group $K=G^\theta$. These groups are not badly
disconnected, for example the component group is an elementary abelian two-group
(this follows from \cite[Proposition 4.42(a)]{knappvogan},
and the fact that it is true for real tori),
and we
need the following property of their representations. 

\begin{lemma}
\label{l:multfree}
Let $\mu$ be an irreducible, finite-dimensional,  representation of $K$. Then the restriction of $\mu$ to 
$K_0$ is multiplicity free.  
\end{lemma}

\begin{proof}
Suppose $\mu_0$ is an irreducible summand of $\mu|_{K_0}$,
and let $K_1=\Stab_{K}(\mu_0)$.
It is enough to show that $\mu_0$ extends to an irreducible
representation $\mu_1$ of $K_1$. For then, by Mackey theory,
$\Ind_{K_1}^K(\mu_1)$ 
is irreducible, so isomorphic to $\mu$, and restricts to
the sum of the distinct irreducible representations
$\{\pi_0^x\mid x\in S\}$, where  $S$ is a set of representatives of $K/K_1$.

Choose Cartan and Borel subgroups $T\subset B_{K_0}$ of $K_0$.
(We can arrange that $B_{K_0}=B\cap K_0$ and $T=H\cap K_0$).   

\begin{lemma}
We can choose elements $x_1,\dots, x_n\in K$ such that:
\begin{enumerate}[(1)]
\item $K=\langle K_0,x_1,\dots, x_n\rangle$;
\item $x_i$ normalizes $B_{K_0}$ and $T$;
\item The $x_i$ commute with each other.
\end{enumerate}
\end{lemma}

\begin{remark}
By a standard argument it is  easy to arrange  (1) and (2), the main
point is (3).
Alternatively, it is well known that we could instead choose the $x_i$ to satisfy (1),
(3) and  that
each $x_i$ has order $2$.
It would be interesting to prove that one can satisfy all four
conditions simultaneously, and perhaps even that conjugation by $x_i$
is a distinguished involution of $K_0$.
\end{remark}

\begin{proof}
Choose $x\in K\backslash K_0$.   Then conjugation by $x$ takes
$B_{K_0}$ to another Borel subgroup, which we may conjugate back to
$B_{K_0}$. So after replacing $x$ with another element in the same
coset of $K_0$ we may assume $x$ normalizes $B_{K_0}$.
Conjugating again by an element of $B_{K_0}$ we may assume $x$
normalizes $T$. By induction 
this gives (1) and (2). 

For (3), it is straightforward to reduce to the case when $G(\R)$ is
simple. Then a case-by-case check shows that $|K/K_0|\le 2$ except in
type $D_n$. Furthermore the only exception is the adjoint group
$PSO(2n,2n)$, in which case the result can be easily checked. 
This is essentially \cite[Proposition 9.7]{ic4}.
\end{proof}

Let $\lambda\in X^*(T)$
be the highest weight of $\mu_0$ with respect to $B_{K_0}$. 
Then $\mu_0^{x_i}$ has highest weight $x_i\lambda$.
So, after renumbering, we may write $K_1=\langle K_0,x_1,\dots,
x_r\rangle$ where $x_i\lambda=\lambda$ for $1\le i\le r$.

Let $V_\lambda$ be the (one-dimensional) highest weight space of
$\mu_0$. The group $T_1=\langle T,x_1,\dots, x_r\rangle$ acts on
$V_\lambda$.
In the terminology of 
\cite[Definition 1.14(e)]{voganOrange}, $T_1$ is a large Cartan
subgroup of $K_1$, and 
\cite[Theorem 1.17]{voganOrange} implies that there is an
irreducible representation $\mu_1$ of $K_1$,
containing the one-dimensional representation $V_\lambda$ of $T_1$.
Then $\mu_1|_{K_0}=\mu_0$.
\end{proof}

\begin{theorem}
\label{t:schur}
Suppose every irreducible representation of $G(\R)$ is self-dual
(see Corollary \ref{c:everyalt}).
If $\pi$ is an irreducible representation then
\begin{equation}
\epsilon(\pi)=\chi_\pi(z(\Ch\rho)).
\end{equation}
Every irreducible representation is orthogonal if
and only if 
$z(\Ch\rho)=1$. This holds  if $G$ is adjoint.
\end{theorem}

\begin{proof}
By Corollary \ref{c:every} every $K$-type is self-dual, 
and $-1\in W(K,H_f)$.
Choose a minimal $K$-type $\mu$. Since $\mu$ is self-dual and has
multiplicity one, $\epsilon(\pi)=\epsilon(\mu)$. 

Let $\mu_0$ be an irreducible summand of $\mu|_{K_0}$. 
By Lemma \ref{l:multfree} $\mu_0$ has multiplicity one. 
If $\mu_0$ is self-dual then $\epsilon(\mu)=\epsilon(\mu_0)$, and 
by Lemma \ref{l:schurfinite}
$\epsilon(\mu_0)=\chi_{\mu_0}(z(\Ch\rho_K))$.
By Lemma \ref{l:zrhoK} this equals $\chi_{\mu_0}(z(\Ch\rho))$.

Suppose $\mu_0$ is not self-dual. 
Since $-1\in W(K,H_f)$, choose a representative   $g\in \Norm_K(H_f)$ of
$-1\in W(K,H_f)$, and let 
$K^\dagger=\langle K,g\rangle$. By Lemma \ref{l:schurfinite2}
$\mu^\dagger=\Ind_{K_0}^{K^\dagger}(\mu_0)$ is irreducible, self-dual,
and  of multiplicity one in $\mu$, so
$\epsilon(\mu)=\epsilon(\mu^\dagger)$. 
Since $\mu_0\not\simeq\mu_0^*$, by Lemma \ref{l:schurfinite},
$\epsilon(\mu^\dagger)=\chi_{\mu_0}(g^2)$.
We can also think of $g$ as a representative of $-1\in W(G,H_f)$. 
Since $G$ (unlike $K$) is (necessarily) connected,
by  Lemma \ref{l:w0}, $g^2=z(\Ch\rho_G)$, so again
$\epsilon(\mu)=\chi_{\mu_0}(z(\Ch\rho_G))$.

As in the proof of Lemma \ref{l:schurconnected}, since $z(\Ch\rho_G)\in Z(G(\R))$,
$\chi_{\mu_0}(z(\Ch\rho_G))=\chi_\pi(z(\Ch\rho_G))$. This completes the proof.
\end{proof}
  
\bibliographystyle{plain}
\def\cprime{$'$} \def\cftil#1{\ifmmode\setbox7\hbox{$\accent"5E#1$}\else
  \setbox7\hbox{\accent"5E#1}\penalty 10000\relax\fi\raise 1\ht7
  \hbox{\lower1.15ex\hbox to 1\wd7{\hss\accent"7E\hss}}\penalty 10000
  \hskip-1\wd7\penalty 10000\box7}
  \def\cftil#1{\ifmmode\setbox7\hbox{$\accent"5E#1$}\else
  \setbox7\hbox{\accent"5E#1}\penalty 10000\relax\fi\raise 1\ht7
  \hbox{\lower1.15ex\hbox to 1\wd7{\hss\accent"7E\hss}}\penalty 10000
  \hskip-1\wd7\penalty 10000\box7}
  \def\cftil#1{\ifmmode\setbox7\hbox{$\accent"5E#1$}\else
  \setbox7\hbox{\accent"5E#1}\penalty 10000\relax\fi\raise 1\ht7
  \hbox{\lower1.15ex\hbox to 1\wd7{\hss\accent"7E\hss}}\penalty 10000
  \hskip-1\wd7\penalty 10000\box7}
  \def\cftil#1{\ifmmode\setbox7\hbox{$\accent"5E#1$}\else
  \setbox7\hbox{\accent"5E#1}\penalty 10000\relax\fi\raise 1\ht7
  \hbox{\lower1.15ex\hbox to 1\wd7{\hss\accent"7E\hss}}\penalty 10000
  \hskip-1\wd7\penalty 10000\box7} \def\cprime{$'$} \def\cprime{$'$}
  \def\cprime{$'$} \def\cprime{$'$} \def\cprime{$'$} \def\cprime{$'$}
  \def\cprime{$'$} \def\cprime{$'$}

\enddocument
\end